\documentclass[reqno,12pt]{amsart}
\usepackage{amsmath}
\usepackage{epsfig} 
\usepackage{graphics} 
\usepackage{amsfonts}
\usepackage{amssymb}
\usepackage{graphicx}
\usepackage[utf8]{inputenc}
\usepackage{enumerate}
\usepackage{hyperref}
\usepackage{comment}

\usepackage[usenames]{color}
\usepackage{xcolor}
\usepackage{amsbsy}
\usepackage{amscd}
\usepackage{array}
\usepackage{mathtools}
\usepackage{ulem}

\newtheorem{theorem}{Theorem}[section]
\newtheorem{corollary}[theorem]{Corollary}

\newtheorem{lemma}[theorem]{Lemma}

\theoremstyle{definition}
\newtheorem{definition}[theorem]{Definition}
\newtheorem{remark}[theorem]{Remark}

\newtheorem{example}[theorem]{Example}

\newcommand{\R}{\mathbb R}
\newcommand{\N}{\mathbb N}
\renewcommand{\d}{\mathrm d}

\newcommand{\eps}{\varepsilon}

 \newcommand{\dis}{\displaystyle}

\newcommand{\beq}{\begin{equation}}
\newcommand{\eeq}{\end{equation}}

\newcommand{\tcv}{\textcolor{violet}}

\title[]{Solvability of Dirichlet boundary value problems governed by non-monotone differential operators}

\author[Anceschi F., Marcelli C., Papalini F.]{Francesca Anceschi, Cristina Marcelli, Francesca Papalini}
\address[F. Anceschi, C. Marcelli, F. Papalini]{Dipartimento di Ingegneria Industriale e Scienze Matematiche,
Universit\`a Politecnica delle Marche, Via Brecce Bianche 12, 60131 Ancona, Italy}
\email{f.anceschi@staff.univpm.it (Anceschi), c.marcelli@staff.univpm.it (Marcelli), f.papalini@staff.univpm.it (Papalini)}


\begin{document}

\begin{abstract}
 	We prove existence results for Dirichlet boundary value problems for equations of the type
	\begin{align*}
	 	\left( \Phi(k(t) x'(t) ) \right)' = f(t, x(t) , x'(t) ) \qquad \text{for a.e. } t \in I:=[0,T] ,
	\end{align*}
	where  $\Phi : J \to \R $ is a generic possibly non-monotone differential operator defined in a open interval \(J\subseteq \R\),
 	$k:I \to \mathbb{R}$, $k$ is measurable with $k(t) >0$ for a.e. $t \in I$ and $f: \mathbb{R}^3 \to \mathbb{R}$ is a Carathéodory function.
 	
Under very mild assumptions, we prove the existence of solutions for suitably prescribed boundary conditions, and we also address 
the study of the existence of heteroclinic solutions on the half-line $[0,+\infty)$.	
 	
\end{abstract}

\maketitle

\section{Introduction} \label{sec:intro}
Let us consider the boundary value problem (b.v.p.) 
\begin{align} \label{PD}
    \begin{cases}
       \left( \Phi(k(t) x'(t) ) \right)' = f(t, x(t) , x'(t) ) \qquad \text{for a.e. } t \in I:=[0,T], \\
       x(0) = \nu_1 , \qquad x(T) = \nu_2
    \end{cases}
\end{align}
where:
\begin{itemize}
	\item[-] $\Phi : J \to \R $ is a continuous operator, not necessarily monotone, defined in an open interval \(J\) 
	(bounded or unbounded); 
	\item[-] $k:I \to \mathbb{R}$ is measurable, with $k(t) >0$ for a.e. $t \in I$;
	\item[-]$f: I\times \mathbb{R}^2 \to \mathbb{R}$ is a Carathéodory function.
\end{itemize}

\medskip

Boundary value problems involving nonlinear ODEs of this type arise in some interesting models for image denoising, starting from 
the pioneering work of Perona and Malik \cite{PeronaMalik}, see forthcoming Example \ref{ex:Perona}, where a non-monotone diffusion of the type
\begin{equation*}
	 \Phi(s):=\dfrac{s}{1+s^2}
\end{equation*}
is considered to enhance edges in image reconstruction while running forward in time. Furthermore, the continuity and local monotonicity assumption for  $\Phi$ 
allow us to deal with possibly periodic diffusions, such as
\begin{align*}
	\Phi(s):= \sin (s), 
\end{align*}
see Example \ref{ex:sin} in the sequel. These periodic type diffusions may arise in the context of control problems of second 
order with a periodic type feedback. 

Furthermore, one could also consider models involving the difference of Laplacians, where $\Phi$, defined as 
\begin{align*}
	\Phi (s) = |s|^\alpha s - |s|^\beta s, \qquad \alpha \ne \beta,
\end{align*}
is a non-monotone operator, which could be useful to model the interaction of competing populations.

\medskip
According to our knowledge, despite a very wide literature concerning b.v.p. as \eqref{PD} involving strictly monotone operators, see for instance  \cite{An, BRR, CaPo, Ca, KFA} and the references therein, up to now no result was available concerning b.v.p. governed by non-monotone differential operators. Additionally, our approach covers, as particular cases, problems with strictly monotone differential operators. 
Therefore, this work can be regarded as
 a new unifying approach to b.v.p. governed by various type of differential operators, whether monotone or not.

More in detail, our results can be applied also to b.v.p. with strictly increasing homeomorphisms 
\begin{align*} 
	\Phi: (-a, a) \to (-b, b), \qquad \text{with } 0 < a, b \leq + \infty,
\end{align*}
arising in various models, e.g., in non-Newtonian
fluid theory, diﬀusion of flows in porous media, nonlinear elasticity, and theory
of capillary surfaces, see e.g. \cite{EV}, \cite{HV}, and, more recently, the modeling of glaciology, see for
instance \cite{BRR}, \cite{GR}, \cite{PR}.

When $a=b= +\infty$, the main prototype for $\Phi$ is the $r$-Laplacian (see forthcoming in Example \ref{ex:lapl}), where $\Phi$ is defined as
\begin{align*}
	\Phi(s) = | s|^{r-2} s, \qquad \text{with } r >1.
\end{align*}
When $a = + \infty$ and $b < + \infty$, the map $\Phi$ is referred to as 
non-surjective or bounded $\Phi$-Laplacian and its main prototype is the mean curvature operator,   
\begin{align*}
	\Phi(s) = \frac{s}{\sqrt{1+s^2}}, \qquad s \in \mathbb{R}.
\end{align*}
see for instance \cite{BM}, \cite{FMP}. On the other hand, when $a < + \infty$, then the $\Phi$-Laplacian becomes singular and 
in this case the main prototype is the relativistic operator 
\begin{align*}
	\Phi(s) = \frac{s}{\sqrt{1-s^2}}, \qquad s \in (-1, 1),
\end{align*}
see for example \cite{BM2}, \cite{CMP}, \cite{Ma}, or its generalization, see for instance \cite{prel}, given by the 
$p-$relativistic operator defined as 
\begin{align*}
	\Phi(s) = \frac{|s|^{p-2} s}{(1-|s|^p )^{1- \frac1p}}, \qquad p >1.
\end{align*} 

\medskip
Our technique relies on a suitable compatibility condition between the right-hand 
side of the equation in \eqref{PD} and the prescribed boundary data. 
So, we obtain the existence of solutions for suitable values of the boundary data. On the other hand, we require only minimal assumptions, since we do not need neither a Nagumo-type condition, nor the existence of a pair of well-ordered upper and lower-solutions, which are usually assumed in the above cited papers in order to obtain the existence of solutions for any boundary data. Hence, our result is a novelty also in the particular context of strictly increasing differential operators (see Corollary \ref{cor:cor} below). 

\medskip
More in detail, in  what follows we also assume
 \beq \label{ip:kappa}
 1/k \in L^p(I) \quad \text{ for some } p \geq 1,
\eeq    
and we use 
 the following notation: 
\begin{align}
    \label{eq:notat}
  \quad 
    \quad k_p := \Big \| \frac1k \Big \|_{L^p(I)}, 
    \quad k_1 := \Big \| \frac1k \Big \|_{L^1(I)} \text{ and }   \quad s^* := \frac{\nu_2 - \nu_1}{k_1}.
\end{align}

We underline that when \(k\) vanishes at some point, then the b.v.p. is singular and may have an interest in itself regardless of the type of differential operator.

\medskip
In light of assumption \eqref{ip:kappa}, it is natural to search for solutions to \eqref{PD} that belong to $W^{1,p}(I)$, and more specifically
are intended in the following sense.
\begin{definition} \label{def:sol}
    A function $x \in W^{1,p}(I)$ is a solution to problem \eqref{PD} if 
    \begin{itemize}
        \item $x(0)=\nu_1$, $x(T)=\nu_2$;
        \item $\Phi \circ (k x') \in W^{1,1}(I)$;
        \item $\left( \Phi(k(t) x'(t) ) \right)' = f(t, x(t) , x'(t) ) $ a.e. in $I$. 
    \end{itemize}
\end{definition}

\medskip
Our main result is the following.

\begin{theorem}
\label{thm:main1}
Let assumption \eqref{ip:kappa} hold, and suppose that  
\(\Phi\) is locally strictly monotone at \(s^*\), where \(s^*\) is the slope defined in \eqref{eq:notat}. 
Let \(J^*\subseteq J\) be an open interval, containing \(s^*\), where \(\Phi\) is strictly monotone.

Assume that there exists a positive function $\psi \in L^1_+(I)$  satisfying
 \beq \label{ip:bound}        
   \Phi(s^*) \pm  2 \| \psi \|_{L^1(I)} \in \Phi(J^*),
  \eeq   
and
\beq \label{ip:norm}
    | f(t,x,y(t)) | \leq \psi(t) \qquad \text{a.e. in } I, 
\eeq
for every  $x\in \R$ and $y\in  L^p(I)$ such that:
\beq \label{ip:stimax} 
\frac{x-\nu_1}{k_1}\in J^* \quad \text{and} \quad
\left|\Phi\left(\frac{x-\nu_1}{k_1}\right)- \Phi(s^*)\right|<2 \| \psi \|_{L^1(I)} ,
\eeq
\beq \label{ip:stimay} 
k(t)y(t)\in J^* \quad \text{and} \quad \left|\Phi(k(t)y(t))- \Phi(s^*)\right|<2 \| \psi \|_{L^1(I)} 
\quad \text{for a.e. } t \in I.
\eeq

Then, there exists a solution $x \in W^{1,p}(I)$ to problem \eqref{PD}, satisfying estimates in \eqref{ip:stimax} (with $x$ replaced by $x(t)$) and
\eqref{ip:stimay} (with $y(t)$ replaced by $x'(t)$).
\end{theorem}

\bigskip	
Of course, if \(\Phi\) is globally strictly monotone with \(\Phi(J)=\R\), then assumptions \eqref{ip:bound} is trivially satisfied. This occurs when  \(\Phi\) is a homeomorphism from \(\R\) to \(\R\) (as the \(r\)-Laplacian operator), or a singular operator (as the relativistic one). So, in these cases assumption \eqref{ip:norm} is sufficient for the existence of solutions and the following result holds, which is a novelty in the context of monotone operators.

\begin{corollary}\label{cor:cor}
Let assumption \eqref{ip:kappa} hold and suppose that  
 \(\Phi:J\to \R\) a continuous, strictly monotone, surjective operator with $s^*\in J$.
Assume that there exists a positive function \(\psi\in L^1_+(I)\), such that \eqref{ip:norm} holds true 
for every  $x\in \R$ and $y\in  L^p(I)$ satisfying \eqref{ip:stimax} and \eqref{ip:stimay} with $J^*$ replaced by $J$.

Then, problem \eqref{PD} admits a solution.
\end{corollary}

 In order to justify the novelty of the previous result,   we present an example (see Example \eqref{ex:lapl}) where Corollary \ref{cor:cor} is applicable, but the classical Nagumo condition does not hold true, so the example is not covered by existence results already known in literature.

\smallskip

Additionally, in the particular case of singular operators, that is homeomorphisms from an open bounded interval to \(\R\) (as  the relativistic operator), the existence result can be further simplified, as follows.

 \begin{corollary}\label{cor:cor2}
Let assumption \eqref{ip:kappa} hold and suppose that  
 \(\Phi:J\to \R\) a continuous, strictly monotone, surjective operator, defined in a bounded open interval \(J\), with $s^*\in J$.
 
 Assume that there exists a positive function \(\psi\in L^1_+(I)\), such that for a.e. \(t\in I\) we have
\beq \label{ip:cor2}
|f(t,x,y(t))|\le \psi(t) \quad \text{ whenever } \frac{x-\nu_1}{k_1}\in J,\ y\in L^P(I), \ k(t)y(t)\in J.\eeq

Then, problem \eqref{PD} admits a solution.
 \end{corollary}

 Lastly, in Section \ref{sec:half} we deal with the b.v.p. on the half-line $[0,+\infty)$. By means of a suitable limit process we 
obtain sufficient conditions for the existence of heteroclinic solutions.

\subsection*{Strategy of the proof and outline} 
A solution to \eqref{PD} is obtained by a fixed-point method applied to a suitable associated functional b.v.p. (see Section \ref{sec:aux}). Section \ref{sec:proof} is devoted to the proof of Theorem   
\ref{thm:main1}. Then, in Section \ref{sec:ex}  we present some examples related to various interesting models. 
Finally, in Section \ref{sec:half} we provide existence results for the Dirichlet problem on the half-line, together with examples of their application.

\section{Auxiliary problem} 
\label{sec:aux}
Let us consider the auxiliary problem 
\begin{align} \label{PD-aux}
    \begin{cases}
       \left( \Phi(k(t) x'(t) ) \right)' = F_x(t) \qquad \text{for a.e. } t \in I=[0,T], \\
       x(0) = \nu_1 , \qquad x(T) = \nu_2,
    \end{cases}
\end{align}
where $F:W^{1,p}(I) \to L^1(I)$ is a continuous operator such that there exists a non-negative function $\psi \in L^1(I)$ satisfying
\beq \label{ip:psiF}
        |F_x(t) | \leq \psi(t) \qquad \text{for a.e. } t \in I, \text{for all } 
        x \in W^{1,p}(I).
\eeq

\noindent
For every $x \in W^{1,p}(I)$, we define the integral operator
\begin{align}
    \label{eq:intop}
    \mathcal{F}_x(t) = \int_0^t F_x(s)  \d s , \qquad t \in I. 
\end{align}
Of course, $\mathcal{F}$ is a continuous operator in $W^{1,p}(I)$
and from assumption \eqref{ip:psiF} it satisfies
\begin{align}
    \label{eq:Fpsi}
    |\mathcal{F}_x (t) | \leq \| \psi\|_{L^1(I)} \qquad \text{for all } t \in I, x \in W^{1,p}(I).
\end{align}

\medskip
We look for solutions to the auxiliary Dirichlet problem \eqref{PD-aux}
according the following definition.
\begin{definition} \label{def:solaux}
    A function $x \in W^{1,p}(I)$ is a solution to problem \eqref{PD-aux}
    if 
    \begin{itemize}
        \item $x(0)=\nu_1$, $x(T)=\nu_2$;
        \item $\Phi \circ (k x') \in W^{1,1}(I)$;
        \item $\left( \Phi(k(t) x'(t) ) \right)' = F_x(t)$ a.e. in $I$. 
    \end{itemize}
\end{definition}

\medskip
The following preliminary result provides a compatibility condition for the solvability of problem \eqref{PD-aux}.

\begin{lemma}
    \label{lem:aux}   
    Let condition \eqref{ip:psiF} hold and assume that  \(\Phi\) is locally strictly monotone at \(s^*\). Moreover, 
   let \(J^*\subseteq J\) be an open interval containing \(s^*\) where \(\Phi\) is strictly monotone, and assume that 
    \eqref{ip:bound} is satisfied.
    
    \smallskip
    Then, for every $x \in W^{1,p}(I)$ there exists a unique $\beta_x \in \mathbb{R}$ such that
    \begin{align} \label{eq:lemma1}
        \int_0^T \frac{1}{k(t)} \Phi^{-1}\left( \beta_x + \mathcal{F}_x(t) \right) 
        \, \d t = \nu_2 - \nu_1
    \end{align}
    where \(\Phi^{-1}\) is the partial inverse of \(\Phi\), defined in \(\Phi(J^*)\) and taking values on \(J^*\).
    Moreover, 
    \beq \label{eq:betadom}
        \Phi(s^*) - \| \psi \|_{L^1(I)} \leq \beta_x \leq \Phi(s^*) + \| \psi \|_{L^1(I)}.
    \eeq
\end{lemma}

\begin{proof} 
We provide the proof in the case \(\Phi\) is increasing in \(J^*\), 
since if \(\Phi\) is decreasing it suffices to replace \(\Phi\) with \(-\Phi\).
Lastly, throughout this proof we denote by \(L:=\|\psi\|_{L^1(I)}\), \(J^*=:(a_1,a_2)\), 
\(\Phi(J^*)=:(b_1,b_2)\), with \(a_1,a_2,b_1,b_2\) possibly infinite.

   Put
       \( \mathcal{I}:= \left( b_1 + L, b_2 -L \right)\) with the obvious meaning of the operations when \(b_1\) or \(b_2\) are infinite.   It is nonempty by assumption \eqref{ip:bound}; moreover, 
 by \eqref{eq:Fpsi},  for every \(x\in W^{1,p}(I)\) we have 
\beq b_1< \xi-L\le \xi +  \mathcal{F}_x(t) \le \xi+ L < b_2 \quad \text{ for every } \xi\in \mathcal{I}, t\in I. \label{eq:chain}\eeq 
So,
for every fixed $x \in W^{1,p}(I)$, we can define the function
\( \varphi_x : \mathcal{I} \to \mathbb{R}\) as
    \begin{align}
        \label{eq:phi}
        \varphi_x(\xi) := \int_0^T \frac{1}{k(t)} \Phi^{-1}
        \left(\xi + \mathcal{F}_x(t) \right)  \d t .
    \end{align}  
Note that the integral in the definition of \(\varphi\) is finite by \eqref{ip:kappa} and \eqref{eq:chain}.
  
   Since $\Phi$ is strictly increasing in \((a_1,a_2)\), also  
     $\Phi^{-1}$ is strictly increasing in $(b_1, b_2)$ and so is $\varphi_x$. 
     Moreover, \(\varphi\) is  
    continuous by the Dominated Convergence Theorem and \eqref{eq:chain}.
    Additionally, for every $\xi \in\mathcal{I}$ we have
    \begin{align*}
        k_1 \Phi^{-1} \left( \xi - L\right)< \varphi_x(t) 
        < k_1\Phi^{-1} \left( \xi + L \right) .
    \end{align*}
   Then,  taking the limits as $\xi \to (b_1 + L)^+$ and  as 
   $\xi \to (b_2 - L)^-$, by \eqref{ip:bound} we get 
   \begin{align*}
        \lim \limits_{\xi \to (b_1 +L)^+} \varphi_x(\xi) 
                        \leq k_1 \Phi^{-1} (b_1 + 2L ) < 
         k_1 \Phi^{-1} \left( b_2 - 2 L \right) \leq  \lim \limits_{\xi \to (b_2 -L)^-} \varphi_x (\xi), 
   \end{align*}
  where \(\Phi^{-1}(\pm \infty)\), if this situation occurs, have to be intended as the limits of \(\Phi^{-1}(\xi)\) as \(\xi\to \pm \infty\).
   
    Hence, by assumption \eqref{ip:bound}, since $\varphi_x$ is strictly increasing and continuous, we get 
    \begin{align*}
        \left( k_1 \Phi^{-1} (b_1 + 2 L ) , k_1 \Phi^{-1} \left( b_2 - 2 L \right) \right) \subset \varphi_x (\mathcal{I}).
    \end{align*}
    Moreover, by assumption \eqref{ip:bound} we also have the following 
    \begin{align*}
       k_1 \Phi^{-1}( b_1 + 2 L) < \nu_2 - \nu_1 < k_1 \Phi^{-1} (b_2 - 2 L),
    \end{align*}
    and hence $\nu_2 - \nu_1 \in \varphi_x(\mathcal{I})$.
    
    Therefore, for every fixed $x \in W^{1,p}(I)$ there exists a unique $\beta_x \in \mathcal{I}$
    such that $\varphi_x (\beta_x) = \nu_2 - \nu_1$, and claim \eqref{eq:lemma1} follows from \eqref{eq:phi}.

    Furthermore, from the Mean Value Theorem, for every $x \in W^{1,p}(I)$ there exists $t_x \in I$ such that 
    \begin{align*}
        k_1 \Phi^{-1} \left( \beta_x + \mathcal{F}_x (t_x) \right) &= \nu_2 - \nu_1,
    \end{align*}
    from which we infer
    \begin{align*}
        \beta_x &= \Phi(s^*) -  \mathcal{F}_x (t_x), 
    \end{align*}
    and thus \eqref{eq:betadom} follows.
\end{proof}

\medskip
\begin{theorem}
    \label{thm:abstract}
    Let assumption \eqref{ip:kappa} hold and suppose that  
\(\Phi\) is locally strictly monotone at \(s^*\), where \(s^*\) is the slope defined in \eqref{eq:notat}. 
Let \(J^*\subseteq J\) be \tcv{an open} interval, containing \(s^*\), where \(\Phi\) is strictly monotone.
 
  Then, problem \eqref{PD-aux} admits a   solution $x \in W^{1,p}(I)$, 
satisfying estimates in \eqref{ip:stimax} (with $x$ replaced by $x(t)$) and
\eqref{ip:stimay} (with $y(t)$ replaced by $x'(t)$).
\end{theorem}

\smallskip

\begin{proof}  Also in this proof we assume with no loss of generality that \(\Phi\) is increasing in the interval \(J^*\), otherwise it suffices to change the sign in the differential equation of problem \eqref{PD-aux}.
Moreover, also throughout this proof we denote by \(L:=\|\psi\|_{L^1(I)}\) and by $\Phi^{-1}$ the partial inverse of $\Phi$.
Lastly, we put 
\[ J^*=:(a_1,a_2), \quad \Phi(J^*)=: (b_1,b_2).\]
Then, let us consider the operator 
    \beq \label{def:g}
        g : W^{1,p}(I) \to W^{1,p}(I) ; \quad
        g_x(t) := \nu_1 + \int_0^t \frac{1}{k(s)} \Phi^{-1} \left( \beta_x + \mathcal{F}_x(s) \right) \d s ,
    \eeq
where \(\mathcal{F}_x\) is defined by \eqref{eq:intop}, and $\Phi^{-1}$ is the partial inverse function of $\Phi$, defined in $\Phi(J^*)$ and taking values in $J^*$.
Notice that if \(x\) is a fixed point for \(g\), then it is a solution to problem \eqref{PD-aux}. In fact, by Lemma \ref{lem:aux}, \(x\) satisfies the boundary conditions. Moreover, for a.e. \(t\in I\) we have
\[ k(t)x'(t)= \Phi^{-1}(\beta_x+ \mathcal{F}_x(t))\]
implying
\[ \Phi(k(t)x'(t))= \Phi(\Phi^{-1}(\beta_x+\mathcal{F}_x(t)))=
\beta_x+\mathcal{F}_x(t). \]
Hence, 
\(\Phi\circ(kx')\) belongs to \(W^{1,1}(I)\) and
\[ (\Phi(k(t)x'(t)))'=F_x(t) \quad \text{ for a.e. } t \in I.\]

We observe that since \(x\in W^{1,p}(I)\) and \(\Phi\circ (kx')\in W^{1,1}(I)\), the continuity of \(\Phi^{-1}\) implies that \(kx'\in C(I)\), meaning that \(kx'\) has a continuous extension on the whole interval \(I\).

Now, we show that the operator \(g\) satisfies all the assumptions of the Schauder's Fixed Point Theorem.

\medskip
   {\bf Step 1.} {\bf $g$ is bounded in $W^{1,p}(I)$.}

\smallskip\noindent
First of all, by \eqref{ip:bound} and \eqref{eq:betadom},  for every \( x\in W^{1,p}(x)\) and  \(t\in I\)  we have
\beq b_1 < \Phi(s^*) -2L \le \beta_x + \mathcal{F}_x(t)\le  \Phi(s^*) +2L <b_2.  \label{eq:codom}\eeq
Hence, the values \(\Phi^{-1}(\Phi(s^*) \pm 2L)\) are both finite. So, if we set
\beq \label{def:AB}
 A^*:=\Phi^{-1}(\Phi(s^*) -2L), \quad 
 B^*:=\Phi^{-1}(\Phi(s^*) +2L)
\eeq 
we have
\beq
\label{eq:domin}
A^*\le \Phi^{-1}(\beta_x+\mathcal{F}_x(t)) \le  B^* \quad \text{
for every } x\in W^{1,p}(I) \text{ and every } t\in I.
\eeq

\smallskip\noindent
Therefore, by \eqref{def:g} for every \(x\in W^{1,p}(I)\)  we have 
\begin{align} \label{eq:claim1}
        \nu_1+k_1A^*\le  g_x(t)\le \nu_1+ k_1B^*, \quad \text{ for every } t \in I, 
    \end{align}
    and
  \beq \label{eq:deriv}
        \frac{A^*}{k(t)}\le g_x'(t)\le \frac{B^*}{k(t)} , \quad \text{ for a.e. } t \in I; 
    \eeq 
   hence, $g$ is bounded in $W^{1,p}(I)$.

    \vspace{4mm}
\medskip
   {\bf Step 2.} {\bf \( g\)  is continuous in $W^{1,p}(I)$.}
  
  \smallskip\noindent 
    Let us fix $x_1, x_2 \in W^{1,p}(I)$ and set $\beta_1 = \beta_{x_1}$ and $\beta_2 = \beta_{x_2}$. Then, 
    \begin{align*}
        \nu_2 - \nu_1 = \int_0^T \frac{1}{k(t)} \Phi^{-1} \left( \beta_1 + \mathcal{F}_{x_1}(t) \right) \, \d t 
        = \int_0^T \frac{1}{k(t)} \Phi^{-1} \left( \beta_2 + \mathcal{F}_{x_2}(t) \right) \, \d t ,
    \end{align*}
    that is
    \begin{align*}
        \int_0^T \frac{1}{k(t)} \left[ \Phi^{-1} \left( \beta_1 + \mathcal{F}_{x_1}(t) \right) 
        - \Phi^{-1} \left( \beta_2 + \mathcal{F}_{x_2}(t) \right) \right] \, \d t = 0 .
    \end{align*}
    Then, from the mean value theorem there exists $\hat{t} \in I$ such that 
    \begin{align*}
        \Phi^{-1} \left( \beta_1 + \mathcal{F}_{x_1}(\hat{t}) \right) = \Phi^{-1} \left( \beta_2 + \mathcal{F}_{x_2}(\hat{t}) \right).
    \end{align*}
    Now, recalling that $\Phi$ is injective we deduce
    \(
        \beta_1 - \beta_2 = \mathcal{F}_{x_2}(\hat{t}) - \mathcal{F}_{x_1}(\hat{t})\). So, 
    \begin{align*}
         |\beta_1 - \beta_2| \leq  \| \mathcal{F}_{x_2} - \mathcal{F}_{x_1} \|_{C(I)}.
    \end{align*}
    But since for every $t \in I$ 
    \begin{align*}
        | \mathcal{F}_{x_2}(t) - \mathcal{F}_{x_1}(t) | \leq \int_0^t | F_{x_1}(s) - F_{x_2}(s) | \d s 
        \leq \| F_{x_1} - F_{x_2} \|_{L^1(I)},
    \end{align*}
     we also get
    \begin{align} \label{eq:beta}
         |\beta_1 - \beta_2| \leq  \| \mathcal{F}_{x_2} - \mathcal{F}_{x_1} \|_{C(I)}
         \leq \| F_{x_1} - F_{x_2} \|_{L^1(I)}.
    \end{align}

    Moreover, by the uniform continuity of \(\Phi^{-1}\) in  \([\Phi(s^*)-2L,\Phi(s^*)+2L]\subset (b_1,b_2)\) (see \eqref{eq:codom}),  for every
    $\varepsilon > 0 $ there exists $\delta = \delta (\varepsilon) > 0$ such that 
    \begin{align} \label{eq:uc}
        | \Phi^{-1}(r_1) - \Phi^{-1}(r_2)| < \min \left\{ \frac{\varepsilon}{2 k_p}, \frac{\varepsilon}{2Tk_1} \right\}
    \end{align}
    for every $r_1, r_2 \in [\Phi(s^*)-2L,\Phi(s^*)+2L]$ with $|r_1 - r_2 | < \delta$.

    Let us now consider a sequence $\{ x_n \}_n \subset W^{1,p}(I)$ such that $x_n \to x$ in $W^{1,p}(I)$, and 
    denote by $\beta_n := \beta_{x_n}$. 
    Since $F: W^{1,p}(I) \to L^1(I)$ is a continuous operator,  there exists $\overline{n}= \overline{n}(x) \in \mathbb{N}$ such that 
    \begin{align*}
        \| F_{x_n} - F_{x} \|_{L^1(I)} < \frac{\delta}{2} \qquad \text{for every $n \geq \overline{n}$ }.
    \end{align*}
    Thus, for every $n \geq \overline n$ and every $t \in I$, by \eqref{eq:beta} it follows
    \begin{align*}
        | \beta_{x_n} + \mathcal{F}_{x_n}(t) - \left( \beta_x - \mathcal{F}_x(t) \right) | &\leq 
        | \beta_{x_n}  - \beta_x | + \int_0^t | F_{x_n}(s) - F_x(s)| \, ds \\
        &\leq 2 \| F_{x_n} - F_x \|_{L^1(I)} < \delta , 
    \end{align*}
    implying that for every $n \geq \overline n$
    and for every $t \in I$ we get by \eqref{eq:uc}
    \begin{align*}
        | g_{x_n}'(t) - g_{x}'(t)| =  \frac{1}{k(t)}\Big|\left( \Phi^{-1}(\beta_{x_n} + \mathcal{F}_{x_n}(t) \right) 
        - \Phi^{-1} \left( \beta_x + \mathcal{F}_{x}(t) \right) \Big | < \frac{1}{k(t)} \frac{\varepsilon}{2 k_p}.
    \end{align*}
    Thus, on the one hand for every $n \geq \overline n$ it is true that 
    \begin{align} \label{eq:gprime}
        \| g_{x_n}' - g_x' \|_{L^p(I)} \leq \frac{\varepsilon}{2 k_p} k_p = \frac{\varepsilon}{2} .
    \end{align}
    On the other hand, again by \eqref{eq:uc} for every $n \geq \overline n$ we have
    \begin{align*}
        | g_{x_n} (t) - g_x (t) | &\leq \int_0^T \frac{1}{k(s)} 
        \Big | \Phi^{-1} \left(\beta_{x_n} + \mathcal{F}_{x_n}(t) \right) 
        - \Phi^{-1} \left( \beta_x + \mathcal{F}_x(t) \right) \Big| \d t \\
        &\leq \int_0^T \frac{1}{k(s)} \frac{\varepsilon}{2 T k_1} \, \d t = \frac{\varepsilon}{2T},
    \end{align*}
    from which it follows that for every $n \geq \overline n$
    \begin{align} \label{eq:g}
        \| g_{x_n} - g_{x} \|_{L^p(I)} < \frac{\varepsilon}{2}. 
    \end{align}
    Hence, combining \eqref{eq:gprime} and \eqref{eq:g} we conclude $\| g_{x_n} - g_{x} \|_{W^{1,p}(I)} \leq \varepsilon$
    for every $n \geq \overline n$, and thus $g:W^{1,p}(I) \to W^{1,p}(I)$ is continuous. 

    \vspace{4mm}
   {\bf Step 3.} {\bf $g$ is a compact operator.}
  
  \smallskip\noindent
    Let $D \subset W^{1,p}(I)$ be a bounded set and $\{ x_n \}_n \subset D$. Our aim is to show that $\{ g_{x_n} \}_n$
    admits a converging subsequence in $W^{1,p}(I)$. 

    First of all, we show that $\{ g_{x_n}' \}_n$ admits a converging subsequence in $L^p(I)$. 
    Note that by \eqref{eq:claim1} for every $s,t \in I$ and every $n \in \mathbb{N}$ we have 
    \begin{align*}
        \Big | \int_s^t g_{x_n}'(\tau) \, d\tau \Big | \leq \Big | g_{x_n} (t) \Big | + \Big | g_{x_n}(s) \Big| 
        \leq 2 | \nu_1| + 2k_1 (|A^*|+|B^*|).
    \end{align*}
     Hence, taking account of \eqref{eq:deriv}, we infer $\{ g_{x_n}' \}_n$ is 
    a uniformly integrable sequence. 

    Then, considering the characterization of relatively compact sets in $L^p$, with $p \geq 1$, (see \cite[Theorem 2.3.6]{GazPap}) we need to show 
    \begin{align*}
        \lim \limits_{h \to 0^+} \int_0^{T-h}  \Big| g_{x_n}' (t+h) - g_{x_n}'(t) \Big | \, \d t = 0 
    \end{align*}
    uniformly in $n \in \mathbb{N}$. 
    In order to do this, observe that since $1/k \in L^p(I)$, with $p \geq 1$, by the Riesz-Kolmogorov Theorem 
    \cite[Theorem 2.3.8]{GazPap} we have
    \begin{align*}
        \lim \limits_{h \to 0^+} \int_0^{T-h} \Big | \frac{1}{k(t+h)} - \frac{1}{k(t)} \Big |^p \, \d t = 0 .
    \end{align*}
    Thus, denoted by $H^*:=\max\{|A^*|, |B^*|\}$ (see \eqref{def:AB}), for every fixed $\varepsilon > 0$, in relation to $\frac{\varepsilon}{(2 H^*)^p}$, there exists $	\rho_1 = \rho_1 (\varepsilon) > 0 $ such that if $0 < h < \rho_1$, then 
    \begin{align}
        \label{eq:tr}
        \int_0^{T-h} \Big | \frac{1}{k(t+h)} - \frac{1}{k(t)} \Big |^p \, \d t < \frac{\varepsilon}{(2 H^*)^p}.
    \end{align}
    Furthermore, from the uniform continuity of $\Phi^{-1}$ in $ [\Phi(s^*)-2L,\Phi(s^*)+2L]$, 
    in relation to $\frac{\varepsilon^{\frac1p}}{2k_p}$, there exists $\delta_1= \delta_1 (\varepsilon) > 0$ such that
    \begin{align}
        \label{eq:circ1}
        | \Phi^{-1}(r_1) - \Phi^{-1}(r_2) | < \frac{\varepsilon^{\frac1p}}{2k_p}
    \end{align}
    if $|r_1 - r_2 | < \delta_1$, with $r_1, r_2 \in  [\Phi(s^*)-2L,\Phi(s^*)+2L]$. 
    Lastly, since $\psi \in L^1(I)$, there exists $\rho_2 = \rho_2(\varepsilon)>0$ such that 
    \begin{align*}
         \int_{\theta_1}^{\theta_2} \psi (t) \, \d t  < \delta_1 
 \quad \text{whenever } \theta_1, \theta_2 \in I,\ |\theta_1 - \theta_2| < \rho_2.    \end{align*} 
 Then, 
    \begin{align} \label{eq:circ2}
        \Big |\mathcal{F}_{x_n} (\theta_1) - \mathcal{F}_{x_n}(\theta_2) \Big | = 
        \Big | \int_{\theta_1}^{\theta_2} F_{x_n}(t) \, \d t \Big | \leq 
         \int_{\theta_1}^{\theta_2} \psi(t) \, \d t  < \delta_1.
    \end{align}
    Combining \eqref{eq:circ1} with \eqref{eq:circ2} we obtain that for every $n \in \mathbb{N}$, it follows 
    \begin{align} \label{eq:circ}
       \Big | \Phi^{-1} (\beta_{x_n} + \mathcal{F}_{x_n}(\theta_1) ) 
       - \Phi^{-1} (\beta_{x_n} + \mathcal{F}_{x_n}(\theta_2) \Big | < \frac{\varepsilon^{\frac1p}}{2k_p} \quad \text{ if } |\theta_1 - \theta_2| < \rho_2.
    \end{align}
    Now,  fix $t \in [0,T)$ and $h\in (0,\rho_2)$ such that $t+h \le T$. Then by \eqref{eq:domin} we have
    \begin{align*}
        &\Big | g_{x_n}'(t+h) - g_{x_n}'(t) \Big | \\ 
        &= 
        \Big | \frac{1}{k(t+h)} \Phi^{-1} \left( \beta_{x_n} + \mathcal{F}_{x_n}(t+h) \right)
        - \frac{1}{k(t)} \Phi^{-1}  \left( \beta_{x_n} + \mathcal{F}_{x_n}(t) \right) \Big | \\
        &\leq \Big | \frac{1}{k(t+h)} -  \frac{1}{k(t)} \Big | \Big | \Phi^{-1}  \left( \beta_{x_n} + \mathcal{F}_{x_n}(t) \right) \Big| \\
        &\qquad + \frac{1}{k(t+h)} \Big|  \Phi^{-1} \left( \beta_{x_n} + \mathcal{F}_{x_n}(t+h) \right)- \Phi^{-1}  \left( \beta_{x_n} + \mathcal{F}_{x_n}(t) \right) \Big| \\
        &\leq \Big | \frac{1}{k(t+h)} -  \frac{1}{k(t)} \Big | \, \, H^* + \\
        & \qquad + \frac{1}{k(t+h)} \Big|  \Phi^{-1} \left( \beta_{x_n} + \mathcal{F}_{x_n}(t+h) \right)- \Phi^{-1}  \left( \beta_{x_n} + \mathcal{F}_{x_n}(t) \right) \Big|.
    \end{align*}
    Observe now that by the convexity of the function \(t\mapsto |t|^p\) we have 
   $|a + b|^p \leq 2^{p-1} |a|^p + 2^{p-1} |b|^p$ and hence by \eqref{eq:tr} and \eqref{eq:circ} we deduce
    \begin{align*}
        &\int_0^{T-h} \Big | g_{x_n}'(t+h) - g_{x_n}'(t) \Big |^p \, \d t \\
        &\qquad \leq 2^{p-1}(H^*)^p \frac{\varepsilon}{(2 H^*)^p} + 2^{p-1} \int_0^{T-h} \frac{1}{k^p(t+h)}  \left( \frac{\varepsilon^\frac1p}{2k_p} \right)^p \, \d t  \\
        &\qquad = \frac{\varepsilon}{2} + 2^{p-1}  \frac{\varepsilon}{2k_p^p} \int_0^{T-h} \frac{1}{k^p(t+h)}  \, \d t\\
        &\qquad \leq  \frac{\varepsilon}{2} + 2^{p-1}  \frac{\varepsilon}{2k_p^p} k_p^p = \varepsilon.
    \end{align*}
    So, we have proved  that 
   \begin{align*}
        \lim \limits_{h \to 0^+} \int_0^{T-h}  \Big| g_{x_n}' (t+h) - g_{x_n}'(t) \Big | \, \d t = 0 .
    \end{align*}
    Hence, the sequence $\{ g_{x_n}' \}_n$ admits a subsequence, by abuse of notation also denoted by 
    $\{ g_{x_n}' \}_n$, which converges in $L^p(I)$ to a function $y \in L^p(I)$ by the characterization 
    of relatively compact sets in $L^p$. 

    Finally,  denoted by 
    \begin{align*}
        z(t) := \nu_1 + \int_0^t y(s) \, ds ,
    \end{align*}
    for $t \in I$, we have
    \begin{align*}
        \Big | g_{x_n}(t) - z(t) \Big | &= \Big | \int_0^t \left( g_{x_n}'(s) - y(s) \right) \, ds \Big | 
        \leq T^{\frac1q} \| g_{x_n}' - y \|_{L^p(I)},
    \end{align*}    
    where in the last line we applied H\"older's inequality with $q=\frac{p}{p-1}$, with the 
    convention that $q = + \infty$ if $p=1$. From this, we infer $g_{x_n}(t) \to z(t)$ uniformly in $I$ and thus 
    $g_{x_n}(t) \to z(t)$ in $L^p(I)$. Since $z'=y$ a.e. it also follows that $\{ g_{x_n} \}_n$ converges in $W^{1,p}(I)$, 
    and in conclusion $g_x$ is a compact operator. 
    
\smallskip    
    By the Schauder Fixed Point Theorem applied to the operator $g$, then 
    it follows there exists a solution to \eqref{PD-aux}.
    Finally, estimates \eqref{ip:stimax} and \eqref{ip:stimay} follow from \eqref{eq:claim1} and \eqref{eq:deriv}, since the solution $x$ is a fixed point of $g$.
\end{proof}

\section{Proof of Theorem \ref{thm:main1}}
\label{sec:proof}
Assume, without restriction, that $\Phi$ is strictly increasing in $J^*$.
   In order to find a solution to \eqref{PD}, for every $\xi, \eta \in L^1(I)$ such that 
    $\xi(t) \leq \eta(t)$ a.e. in $I$ we introduce a truncating operator $\mathcal{T}^{\xi, \eta} : L^1 (I) \to L^1(I)$
    such that 
    \begin{align}
        \label{eq:defT}
        \mathcal{T}_x^{\xi, \eta}(t) &= \max \left\{ \xi(t), \min \{ x(t),\eta(t) \} \right\} , 
        \qquad t \in I.
    \end{align}
    Now, we set
\beq\label{def:ABL} L:=\|\psi\|_{L^1(I)}, \quad
 A^*:=\Phi^{-1}(\Phi(s^*) -2L), \quad 
 B^*:=\Phi^{-1}(\Phi(s^*) +2L),
\eeq     
and 
    \begin{align}
        \label{eq:def}
         N_1 :=  \nu_1  + k_1 A^*, \  N_2 :=  \nu_1  + k_1B^*, \quad \eta_1 (t):= \frac{A^*}{k(t)}, \  \eta_2 (t):= \frac{B^*}{k(t)} ,
    \end{align}
    with the notations above we define a function $F: W^{1,p}(I) \to L^1(I)$, $u \mapsto F_u$, such that 
   \begin{align*}
        F_u(t) = f  \left( t, \mathcal{T}_u^{N_1, N_2}(t), \mathcal{T}_{u'}^{\eta_1, \eta_2} (t) \right) 
        \qquad \text{for every } t \in I. 
    \end{align*}
    Hence, we are now in a position to consider the truncated problem associated to \eqref{PD}, and defined as
    \begin{align}
        \label{PD-trun}
        \begin{cases}
        \left( \Phi(k(t) x'(t) ) \right)' = F_u(t) \qquad \text{a.e. in } I,  \\
        u(0)=\nu_1, \qquad \nu_2(T)=\nu_2.
        \end{cases}
    \end{align}
     We need to show that $F_u$ satisfies the assumptions of Theorem \ref{thm:abstract}. 
     First, we observe that by \eqref{ip:norm}
     \begin{align*}
         |F_u(t)| = \Big | f \left( t,  \mathcal{T}_u^{N_1, N_2}(t), \mathcal{T}_{u'}^{\eta_1, \eta_2} (t) \right)  \Big | ,
     \end{align*}
     and then 
     \begin{align} \label{eq:Fbound}
         |\mathcal{F}_u(t)| = \Big | \int_0^t F_u(s) \, ds \Big | \leq \| \psi \|_{L^1(I)}. 
     \end{align}
     Furthermore, $F$ is a continuous operator from $W^{1,p}(I)$ into $L^1(I)$. Indeed, given a sequence 
     $\{ u_n \} \subset W^{1,p}(I)$ such that $u_n \to u$ in $W^{1,p}(I)$, by possibly passing to a subsequence we have
     \begin{align*}
         u_n \to u \, \, \text{in } W^{1,1}(I) \quad \text{and} \quad u_n' \to u' \, \, \text{in } L^1(I).
     \end{align*}
     Then, by \cite[Lemma A.1]{BCMP}, we get 
     \begin{align*}
         \mathcal{T}_{u_n}^{N_1,N_2} \to \mathcal{T}_{u}^{N_1,N_2} \, \, \text{in } W^{1,1}(I) \quad \text{and} \quad
         \mathcal{T}_{u_n'}^{\eta_1, \eta_2}  \to \mathcal{T}_{u'}^{\eta_1, \eta_2}  \, \, \text{in } L^1(I).
     \end{align*}
     Then, by the above relations and the fact that $f$ is a Carathéodory function, for a.e. $t \in I$ 
     we obtain the following
     \begin{align*}
         \lim \limits_{n \to + \infty} F_{u_n}(t) =
         \lim \limits_{n \to + \infty} f  \left( t, \mathcal{T}_u^{N_1, N_2}(t), \mathcal{T}_{u'}^{\eta_1, \eta_2} (t) \right) 
         = F_u(t). 
     \end{align*}
     In conclusion, by combining this point-wise result with a standard dominated convergence argument based on 
     the bound \eqref{eq:Fbound} we conclude $F_{u_n} \to F_u$ in $L^1(I)$, as $n \to + \infty$, which 
     concludes the proof of the continuity of the operator. 

     Since the assumptions of Theorem \ref{thm:abstract} are satisfied, problem \eqref{PD-trun} admits a solution, which is  a fixed point of the operator \(g\) (see \eqref{def:solaux}). Hence, there exists a function $u \in W^{1,p}(I)$ such that 
     $ u(0)= \nu_1$, $u(T)= \nu_2$, and by \eqref{eq:deriv},  \eqref{def:ABL} and \eqref{eq:def} we get
      \[  
      \eta_1(t)= \frac{A^*}{ k(t) }\le u'(t) \leq   \frac{B^*}{ k(t)}=\eta_2(t) 
      \qquad \text{for a.e. } t \in I.
      \] 
      So, 
   $\mathcal{T}_{u'}^{\eta_1, \eta_2}  (t) = u'(t)$ a.e. (see \eqref{eq:defT}). 
   Furthermore, since
    \(u(t) = \nu_1 +  \int_0^t u'(s) \d s\) for every $ t \in I,$ we have
     \begin{align*}
   N_1=   \nu_1 +   k_1 A^* \le  u(t)  \leq  \nu_1+ k_1B^*   = N_2 \qquad \text{for every } t \in I.
     \end{align*}
     This implies that $\mathcal{T}_{u}^{-N_1, N_2}(t) = u(t) $ for every $t \in I$, and in conclusion $F_u(t) = f(t,u(t), u'(t))$ a.e. in $I$. So, $u$ is a solution to \eqref{PD}. 

\hfill $\Box$

\begin{remark}\label{rem:odd}
Notice that if $\Phi$ is an odd, strictly monotone operator in $J^*$, a symmetric interval with respect to $0$, then 
 conditions \eqref{ip:stimax} and \eqref{ip:stimay} are respectively equivalent to the following ones
\begin{align} \label{ip:stimax_odd}
|x|&<|\nu_1|+k_1\Phi^{-1}(\Phi(|s^*|) +2\|\psi\|_{L^1(I)}); \\
 |y(t)|&<\frac{1}{k(t)}\Phi^{-1}(\Phi(|s^*|) +2\|\psi\|_{L^1(I)})
  \qquad \text{for a.e. } t \in I.  \label{ip:stimay_odd}
\end{align}
\end{remark}

\section{Examples}
\label{sec:ex}
In this section we show by means of some example how Theorem \ref{thm:main1} can be usefully applied in order to get solutions to boundary value problems involving various nonlinear differential operators.

\medskip

\begin{example} \label{ex:Perona}
Let us consider the Perona-Malik operator \(\Phi(s):=\dfrac{s}{1+s^2}\), and the problem
\beq \label{eq:ese1}
	\begin{cases} 
 		\left( \frac{ x'(t)}{1+(x'(t))^2} \right)' = t^\alpha g(x(t)) h(x'(t))   \qquad \text{for a.e. } t \in [0,1], \\
        		x(0) = 0 , \qquad x(1) = \lambda, 
      	\end{cases}
\eeq
where \(\alpha>0\), \(|\lambda|<1\) and \(h,g\) are continuous functions in \(\R\), with \(g\)   bounded.
In this case, we have \(s^*=\lambda\) and we take 
\(J^*=(-1,1)\) so that \(\Phi(J^*)=(-\frac12,\frac12)\).

Now, we set \(M:=\max_{y\in [-1,1]}| h(y)|\) and \(N:=\|g\|_{L^\infty(\R)}\),
 then we introduce \(\psi(t):=M N t^\alpha\). Then,  inequality \eqref{ip:norm} is satisfied for every $x\in \R$ and every $y\in L^p(I)$ such that $y(t)\in J^*$ a.e. $t$.  Moreover, as regards to 
 \eqref{ip:bound}, notice that it holds true  if and only if
\[
  	-\frac12 + \frac{2MN}{\alpha+1} < \Phi(\lambda) < \frac12 -\frac{2MN}{\alpha+1}
\]
that is, taking the symmetry of \(\Phi\) into account,
 \beq
	\label{eq:Perona} 
	\Phi(|\lambda|) < \frac12 -  \frac{2MN}{\alpha+1}
\eeq
and this condition is satisfied when $\frac{MN}{\alpha+1}<\frac14$ and \(|\lambda|\) is small enough.
For instance, if \(M, N\le 1\), then we by requiring \(\alpha>3\) we have \eqref{ip:bound}. 

Notice that the upper bound in condition \eqref{eq:Perona} does not depend on the behaviour of function \(h(y)\) for \(|y|>1\).
\end{example}

\bigskip
\begin{example} \label{ex:sin}
 Let us consider again problem \eqref{eq:ese1} replacing the Perona-Malik operator with an oscillating one, of the type \(\Phi(s)=\sin s\), that is 
\begin{equation*} 
	\begin{cases}  
	\left( \sin( x'(t) ) \right)' = t^\alpha g(x(t)) h(x'(t))   \qquad \text{for a.e. } t \in [0,1], \\
       x(0) = 0 , \qquad x(1) = \lambda ,
       \end{cases}
\end{equation*}
with $\alpha, \lambda \in \mathbb{R}$. 
In this case, if \(|\lambda|<\frac\pi{2}\) we take 
\(J^*=(-\frac\pi{2},\frac\pi{2})\), so that \(\Phi(J^*)=(-1,1)\).
Therefore,  set \(M:=\max_{|y|\le  \frac\pi{2}} |h(y)|\)  (and again \(N=\|g\|_{L^\infty(\R)}\), $\psi(t)=MNt^\alpha$),  inequality \eqref{ip:norm} is satisfied for every $x\in \R$ and every $y\in L^p(I)$ such that $y(t)\in J^*$ a.e. $t\in [0,1]$. Moreover,   assumption \eqref{ip:bound} is fulfilled if 
\beq
\label{eq:ese2}
\sin(|\lambda|) < 1- \frac{2MN}{\alpha+1}.
\eeq
This condition is satisfied whenever \(\lambda\) is sufficienly far from the extrema \(\pm\frac\pi{2}\).
For instance, if \(M,N\le 1\), by requiring \(\alpha>1\), then 
condition \eqref{eq:ese2} becomes \(|\lambda|<\arcsin (1-\frac{2}{\alpha+1})\). 

On the other hand, if $|\lambda|\geq \frac\pi{2}$, one needs to choose accordingly the set $J^*$. For instance, if \(\lambda\in (\frac\pi{2},\frac32 \pi)\), we then choose \(J^*=(\frac\pi{2},\frac32 \pi)\), and the reasoning proceeds as before, replacing \(M\) with \(M':=\max_{\frac\pi{2}\le y\le \frac32\pi} |h(y)|\).

\end{example}

\begin{example} \label{ex:lapl}
Consider now the following autonomous problem involving the p-Laplacian operator
\beq \label{pr:lapl}
	\begin{cases}
		(|x'(t)|^{p-2}x'(t))' = g(x(t)) |x'(t)|^\beta \qquad \text{for a.e. } t \in [0,1], \\
		x(0)=0, \ x(1)=\lambda.
	\end{cases}
\eeq
with \(p>1\), \(\beta \ge 0\), $\lambda \in \mathbb{R}$, and \(g\) a bounded continuous function.

\smallskip
Then, we set \(N:=\|g\|_{L^\infty}\) and we consider the function \(\ell(z):= (\frac{z}{N})^{\frac{p-1}{\beta}}-2z\), 
defined for \(z\ge 0\). 
If \(\beta<p-1\), we have that \(\ell(z)\to +\infty\) as \(z\to +\infty\). So, for every fixed \(\lambda \in \R\) there exists a value \(\bar z\) such that 
\((|\lambda|^{p-1}+2\bar z)^\frac{\beta}{p-1}<\frac{\bar z}{N}\).
So, if we define \(\psi(t):= \bar z\) (a constant function), then
\(\|\psi\|_{L^1(I)}=\bar z\), and 
\[ 
	|g(x)| |y|^\beta \le \bar z \quad \text{ whenever } 
	|y|<(|\lambda|^{p-1} +2\bar z)^\frac{1}{p-1}=\Phi^{-1}(\Phi(|\lambda|)+2\bar z).
\]
Therefore, if \(\beta<p-1\), all the assumptions of Corollary \ref{cor:cor} are fulfilled (see also Remark \ref{rem:odd}) and problem \eqref{pr:lapl} admits solutions for every \(\lambda \in \R\).

\medskip
Instead, if \(\beta>p-1\),  there exists a value \(\delta>0\) such that \(\ell(z)>0\) in \((0,\delta)\). So, if \(|\lambda|^{p-1}\le \dis\max_{z\ge 0} \ell(z)\), we have
\((|\lambda|^{p-1}+2\bar z)^\frac{\beta}{p-1}\le\frac{\bar z}{N}\) for some \(\bar z\) and we can proceed as in the previous case to deduce the existence of a solution to problem \eqref{pr:lapl}.
For instance, if  the differential equation is \(x''=\cos (x(t)) x'(t)^4\) then we have  \(p=2\),  \(N=1\) and \(\beta=4\), so that 
\(\dis\max_{z\ge 0} \ell(z)=3/8\), and problem \eqref{pr:lapl} admits solution whenever \(|\lambda|\le 3/8 \).

\smallskip
We point out that when \(\beta>p\), then the classical Nagumo condition, required in all the existence results available in the literature (see, e.g., \cite[Remark 4]{fp}), does not hold.  So, this example shows that Theorem \ref{thm:main1} is new even in the particular case of strictly monotone differential operators.
\end{example}

\medskip

We conclude this section with an example involving a singular operator.
\begin{example}
Consider a problem involving the relativistic operator: 
\beq \label{pr:tanh}
\begin{cases}
\left(\dfrac{ x'(t)}{\sqrt{1-x'(t)^2}}\right)' = g(t) h(x(t),x'(t))   \qquad \text{for a.e. } t \in [0,1], \\
x(0)=0, \ x(1)=\lambda,
\end{cases}\eeq
where $|\lambda|<1$, \(g\in L^1([0,1])\) and \(h\) is a continuous function in \(\R^2\).
In this case, the operator \(\Phi\) is strictly monotone in the interval \(J=(-1,1)\).
Therefore, we set \(M:=\max_{(x,y)\in [-1,1]^2} |h(x,y)|\) and \(\psi(t):=Mg(t)\), we have that assumption \eqref{ip:cor2} is  fulfilled and, as a consequence of Corollary \ref{cor:cor2}, we conclude that problem \eqref{pr:tanh} admits a solution for every $\lambda\in (-1,1)$.
\end{example}

\section{Boundary value problems on a half-line}
\label{sec:half}

In this section we consider b.v.p. on a half-line, that is

\begin{align} \label{PD-semir}
    \begin{cases}
       \left( \Phi(k(t) x'(t) ) \right)' = f(t, x(t) , x'(t) ) \qquad \text{for a.e. } t \in \R^+:=[0,+\infty) \\
       x(0) = \nu_1 , \qquad x(+\infty) = \nu_2.
    \end{cases}
\end{align}

where the same assumptions on $\Phi$, $k$ and $f$ of Section \ref{sec:intro} hold, i.e.
\begin{itemize}
	\item[-] $\Phi : J \to \R $ is a continuous operator, defined in an open interval \(J\) (bounded or unbounded); 
	\item[-] $k:\R^+ \to \mathbb{R}$ is  a measurable function such that $k(t) >0$ for a.e. $t>0$;
	\item[-] $f: \R^+\times \mathbb{R}^2 \to \mathbb{R}$ is a Carathéodory function.
\end{itemize}

\smallskip
Furthermore, along this section we will assume 
\beq\label{ip:kL1}
\frac{1}{k}\in L^1(\R^+).
\eeq

\medskip
We obtain the existence of a solution to \eqref{PD-semir} by means of a limit process considering the b.v.p. \eqref{PD} on the sequence of intervals \([0,n]\), as \(n \to +\infty\). In what follows we consider these quantities: 
\beq k_t:=\left\|\frac{1}{k}\right\|_{L^1([0,t])} \quad k_{\infty}:=\left\|\frac{1}{k}\right\|_{L^1(\R^+)} \quad s_t^*:=\frac{\nu_2-\nu_1}{k_t} \quad
 s_\infty^*:=\frac{\nu_2-\nu_1}{k_{\infty}}.\label{eq:not}\eeq
Of course, \(k_t\nearrow k_\infty\) and \(|s_t^*|\searrow |s_\infty^*|\) as $t \to +\infty$.

\medskip
The limit process will lead us to a solution of problem \eqref{PD-semir} thanks to the following result.

\medskip
\begin{lemma}\label{lem:limite} Let \eqref{ip:kL1} holds true. 

Let $x_n \in W^{1,1}([0,n])$ be a solution to b.v.p. \eqref{PD} for \(I=[0,n]\). Suppose that there exist a constant $C\in \R$, a compact interval $ \mathcal{K}\subset J $ and a function $\psi\in L^1 (\R
^+)$ such that $\Phi$ is strictly monotone in $ \mathcal{K} $,
\beq \label{ip:estim} |x_n(t)-\nu_1|\le C, \quad k(t)x_n'(t)\in \mathcal{ K} , \text{ for a.e. } t\in [0,n], \, \text{for every }  n\in \N, \eeq
and
\beq \label{ip:unbound1}
	\left|\left( \Phi(k(t)x_n'(t)\right)^\prime\right|\le \psi(t) \ \ \text{ for a.e. } t\in [0,n], \, \text{for every }  n\in \N.
\eeq
Then, problem \eqref{PD-semir} admits a solution \(x\in W_{\text{loc}}^{1,1}([0,+\infty))\), satisfying the same limitations as in \eqref{ip:estim}.
\end{lemma}

\begin{proof}
	Let $I_n = [0,n]$, with $n \in \mathbb{N}$. 
We extend the functions $x_n$ in $[0,+\infty)$ by defining $x_n(t)=\nu_2$ for every $t\ge n$.
Our aim is to show that the sequence $(x_n)_n$  admits a subsequence uniformly convergent  to a function 
	$x_\infty \in W^{1,1}_{loc}(\mathbb{R}^+)$, which is a solution to \eqref{PD-semir} and satisfies the same limitations as in 
	\eqref{ip:estim}. 

\smallskip

Then, we set
\[w_n(t):= 	\left( \Phi(k(t) x_n'(t)) \right)',\]   extend the functions $w_n$ in $[0,+\infty)$ by defining $w_n(t)=0$ for every $t\ge n$.	
Moreover, we put
$M:=\sup\{|s|:  s\in \mathcal{K}	\}$,
	by the second condition in \eqref{ip:estim} we deduce that 
		\beq\label{eq:xprime} |x_n'(t) | \leq \frac{M}{k(t)} \quad \text{ for a.e. } t \ge 0  \text{ and every } n \in \mathbb{N}. \eeq
 		
	Then, taking also into account \eqref{ip:unbound1}, we can
	  apply the Dunford-Pettis Theorem, to deduce that there exist two functions 
	$z, w \in L^1(\mathbb{R}^+)$
	such that (by eventually passing to a subsequence)
	\begin{align*}
		x_n' \rightharpoonup z \qquad \text{and} \qquad w_n \rightharpoonup w \qquad \text{in } L^1(\mathbb{R}^+).
	\end{align*}
	Hence, for every $t \in \mathbb{R}^+$
	\[
		x_n(t) = x_n(0) + \int_0^t x_n'(s) \, \d s  \xrightarrow[]{}  \nu_1 + \int_0^t z(s) \, \d s =: x_\infty(t) .
	\]
	So, $x_\infty$ is absolutely continuous on $\mathbb{R}^+$ and $x_\infty(0)=  \nu_1$.
	Furthermore, since $
		x_\infty'= z $, we have  $x_\infty' \in L^1(\mathbb{R}^+)$, hence $x_\infty \in W^{1,1}_{\text{loc}}(\mathbb{R}^+)$.
	
	 Now, for every fixed $n \in \N$ we set
	 \begin{align*}
	 	\xi_n (t):=k(t) x_n'(t), \qquad t \in I_n,
	\end{align*}
	and  observe that by \eqref{ip:estim} we have $\xi_n(t)\in   \mathcal{K}$, for a.e. $t\in I_n$ and every $n\in\N$. So, the sequence $(\xi_n(0))_n$ admits a subsequence (labeled again $(\xi_n(0))_n$) converging to some $\xi^*\in \mathcal{K}$.

Now,
	for a.e. every $t \in \mathbb{R}^+ $ there exists $n=n(t) \in \mathbb{N}$ such that 
	for every $n \geq n(t)$ we have
	\begin{align*}
		\Phi( k(t) x_n'(t) ) = \Phi( \xi_n(t) ) = \Phi ( \xi_n(0)) + \int_0^t w_n(s) \, \d s .
	\end{align*}
	Then, since $w_n \rightharpoonup w \in L^1(\mathbb{R}^+)$, we also have
	\begin{align}  \label{eq:V}
		\Phi(k(t) x_n'(t) ) \xrightarrow[]{}  \Phi (\xi^*) + \int_0^t w(s) \d s \in \Phi(	 \mathcal{K}), \quad \text{ for a.e. } t\ge 0.
	\end{align}
	Hence, because of the continuity of the partial inverse $\Phi^{-1}$, defined in the compact $\Phi(\mathcal{K})$, 
	for a.e. $t \ge 0$ we have
	\begin{align*} 
		k(t) x_n'(t) \xrightarrow[]{}  \Phi^{-1} \left(  \Phi (\xi^* ) + \int_0^t w(s) \, ds \right) := v(t).
	\end{align*}
Of course, $v \in C(\mathbb{R}^+)$ and $\Phi \circ v$ is 
	absolutely continuous in $\mathbb{R}^+$. Furthermore, $(\Phi \circ v )' = w \in L^1(\mathbb{R}^+)$.
	Additionally, \eqref{eq:V} implies 
	\begin{align*}
		 x_n'(t) \xrightarrow[]{} \frac{v(t)}{k(t)}  \qquad \text{a.e. in } \mathbb{R}^+.
	\end{align*}
	Then, by \eqref{eq:xprime}, we can apply the Dominated Convergence Theorem and obtain
	\begin{align*}
		x_n' \stackrel{L^1(\mathbb{R}^+)}{\longrightarrow}   \frac{v}{k},
	\end{align*}
	that coupled with the fact that $x_n' \rightharpoonup z \in L^1(I)$, implies that
	\begin{align*}
		x_\infty '(t)= z(t) = \frac{v(t)}{k(t)} \qquad \text{a.e. in } \mathbb{R}^+,
	\end{align*}
	 so, $x_\infty\in W^{1,1}_{\text{loc}}(\mathbb{R}^+)$, $x_n'(t)\to x_\infty'(t)$ a.e. and  by \eqref{eq:V} we have
\[\left(\Phi(k(t)x_\infty'(t))\right)^\prime= 	  w(t) \quad \text{ for a.e. } t\ge 0. \] 
	 
Now, given that 
\[w_n(t) = \left( \Phi ( k(t) x_n'(t) )\right)' = f(t, x_n(t), x_n'(t) ) \quad \text{for a.e. } t>0,
	\]
and $f$ is  a Carathéodory function, we get
\begin{align*}
	 w_n(t) \xrightarrow[]{} f(t, x_\infty(t), x_\infty'(t) ) \qquad \text{for a.e. } t \in \mathbb{R}^+. 
\end{align*}
Moreover, by \eqref{ip:unbound1} we can apply the Dominated Convergence Theorem, to deduce
\begin{align*}
	 w_n(\cdot) \stackrel{L^1(\mathbb{R}^+)}{\longrightarrow}  f(\cdot, x_\infty(\cdot), x_\infty'(\cdot) ) 
\end{align*}
that combined with the fact that $w_n \rightharpoonup w = \left( \Phi ( k \circ x_\infty' ) \right) ' $  guarantees that
\begin{align*}
	\left( \Phi \left( k(t)  x_\infty'(t) \right) \right)' =  f(t, x_\infty(t), x_\infty'(t) ) 	 \qquad \text{a.e. in } \mathbb{R}^+.
\end{align*}
This implies $x_\infty$ is a solution of the equation in \eqref{PD-semir}.	

As far as we are concerned with boundary conditions, we already have that $x_\infty(0)=\nu_1$, since $x_n(0)=\nu_1$ for every $n \in \mathbb{N}$. Moreover, since  
\begin{align*}
	\sup \limits_{t\in\mathbb{R}^+} | x_n(t) - x_\infty(t)| \leq \| x_n' - x_\infty' \|_{L^1(\mathbb{R}^+)},
\end{align*}
and $x_n' \to x_\infty'$ in $L^1(\mathbb{R}^+)$, it follows that
 $x_n \to x_\infty$ uniformly in $\mathbb{R}^+$. Then, 
\begin{align*}
	\lim \limits_{t \to + \infty} x_\infty(t) = \lim \limits_{t \to + \infty} \left( \lim \limits_{n \to + \infty} x_n(t) \right) = \nu_2.
\end{align*}
\end{proof}

\medskip

In the present context of a b.v.p. on a half-line,  we assume that \(\Phi\) is locally strictly monotone at \(s_\infty^*\), that is there exists an open interval \(J^*\subseteq J\) such that
\beq \label{ip:monot0}
 s_\infty^*\in J^*  \quad \text{ and } \quad \Phi \text{ is strictly monotone in } J^*.
\eeq
In the sequel \(\Phi^{-1}\) stands for the partial inverse of \(\Phi\), defined in \(\Phi(J^*)\) and taking values in \(J^*\).

\begin{theorem} \label{th:semir}
Let  \eqref{ip:kL1} and \eqref{ip:monot0}  be satisfied.

Assume that \(\Phi\) is locally Lipschitz at \(s_\infty^*\), that is there exists positive constants \(L,\delta\) such that

\beq \label{ip:lip}
|\Phi(s)-\Phi(s_\infty^*)| \le L|s-s_\infty^*| \quad \text{ whenever } |s-s_\infty^*|<\delta. 
\eeq

Suppose that there exists a function \(\psi\in L^1(\R^+)\) satisfying the following conditions:  
 \beq \label{ip:bound2}        
   \Phi(s_\infty^*) \pm  2 \| \psi \|_{L^1(\R^+)} \in \Phi(J^*),
  \eeq   
\beq \label{ip:psiH}
\text{there exists } \lim_{t \to +\infty} \psi(t)k(t)=:M > L\frac{|\nu_2-\nu_1|}{2k_\infty^2};
\eeq  
and finally
\beq \label{ip:norm2}
    | f(t,x,y(t)) | \leq \psi(t) \qquad \text{a.e. in } I, \eeq
for every  $x\in \R$ and $y\in   L^1(\R^+)$ such that:
\beq \label{ip:stimax2} 
\frac{x-\nu_1}{k_\infty}\in J^* \quad \text{and} \quad
\left|\Phi\left(\frac{x-\nu_1}{k_\infty}\right)- \Phi(s_\infty^*)\right|\le 2 \| \psi \|_{L^1(\R^+)} ,
\eeq
\beq \label{ip:stimay2} 
k(t)y(t)\in J^* \quad \text{and} \quad \left|\Phi(k(t)y(t))- \Phi(s_\infty^*)\right|<2 \| \psi \|_{L^1(\R^+)} 
\quad \text{for a.e. } t \in I.
\eeq

Then, problem \eqref{PD-semir} admits a solution.

\end{theorem}

\begin{proof}
We provide the proof for \(\Phi\) locally strictly increasing at \(s_\infty^*\), since when \(\Phi\) is decreasing it is sufficient to change the sign to both the terms of the differential equation.

First of all, notice that by \eqref{ip:bound2}  and the continuity of \(\Phi\),  there exists  \(\tau_1\) such that  for every 
\(t>\tau_1\) we have:  \(s_t^*\in J^*\), $|s_t^*-s_\infty^*|<\delta$ (see \eqref{eq:not}), 
and
\begin{equation*} 
\Phi(s_t^*) \pm 2 \|\psi\|_{L^1(\R^+)} \in \Phi(J^*).
\end{equation*}

Of course, since $\|\psi\|_{L^1([0,t])}\le \|\psi\|_{L^1(\R^+)}$, 
assumption \eqref{ip:bound2} also implies the validity of \eqref{ip:bound} for \(I=[0,t]\), for every \(t\ge \tau_1\).

\smallskip 
From now on, let us denote 
\[\ell_\infty:=\|\psi\|_{L^1(\R^+)} \quad \text{ and } \quad
\ell_t:=\|\psi\|_{L^1([0,t])}, \ t\ge \tau_1.\]

 Observe that \(\ell_t \nearrow \ell_\infty\) and by \eqref{ip:psiH} we have 
 \[ \lim_{t \to +\infty} \frac{\ell_\infty -\ell_t}{k_\infty-k_t} =
 \lim_{t \to +\infty} \frac{\int_t^{+\infty} \psi(\xi)\d \xi }{\int_t^{+\infty} \frac{1}{k(\xi)} \d \xi}= \lim_{t \to +\infty} \psi(t)k(t) =M.\]
 Let \(\eps^*\) be such that  \(M-\eps^*>L\dfrac{|\nu_2-\nu_1|}{2k_\infty^2}\).
Then,  there exists a value \(\tau_2\ge \tau_1\) such that 
\beq \label{eq:c1}
2\frac{\ell_\infty-\ell_t}{k_\infty-k_t} > 
 L\dfrac{|\nu_2-\nu_1|}{k_\infty^2} +\eps^* \quad \text{ for every } t>\tau_2.
\eeq
Moreover, since \(\frac{k_\infty}{k_t} \to 1\), there exists a value \(\tau_3\ge \tau_2\) such that 
\[
L|\nu_2-\nu_1|\frac{k_\infty}{k_t} <  L|\nu_2-\nu_1| +\eps^*k_\infty^2 \quad \text{ for every } t\ge \tau_3.
\]
Hence, 
\beq \label{eq:c2} L\frac{|\nu_2-\nu_1|}{k_\infty^2} > L\frac{|\nu_2-\nu_1|}{k_t k_\infty} -\eps^* \quad \text{ for every } t\ge \tau_3.\eeq
Thus, by \eqref{eq:c1} and \eqref{eq:c2} we deduce
\[
2\frac{\ell_\infty-\ell_t}{k_\infty-k_t} > 
 L\dfrac{|\nu_2-\nu_1|}{k_tk_\infty}  \quad \text{ for every } t>\tau_3.
\]
 implying by \eqref{ip:lip}
 \[ 2(\ell_\infty-\ell_t)> L|\nu_2-\nu_1| \left( \frac{1}{k_t}-\frac{1}{k_\infty} \right) =  L|s_t^*- s_\infty^*|\ge  |\Phi(s_t^*) - \Phi(s_\infty^*)|\quad \text{ for every } t>\tau_3.
 \]
 
Therefore, we have
\[ \Phi(s_\infty^*)-2\ell_\infty \le \Phi(s_t^*)-2\ell_t \le \Phi(s_t^*)+2\ell_t \le \Phi(s_\infty^*)+2\ell_\infty \]
and then by \eqref{ip:stimay2}
\[ 
 \left|\Phi(k(t)y(t))- \Phi(s_t^*)\right|<2 \ell_t \quad \Rightarrow  
 \quad \left|\Phi(k(t)y(t))- \Phi(s_\infty^*)\right|<2 \ell_\infty.
 \]
Moreover,
\[ 
\left|\Phi\left(\frac{x-\nu_1}{k_t}\right)- \Phi(s_t^*)\right|<2 \ell_t\ \quad \Rightarrow  
 \quad
\left|\Phi\left(\frac{x-\nu_1}{k_\infty}\right)- \Phi(s_\infty^*)\right|\le 2 \ell_\infty.\]
So, for every  \(n\ge \tau_3\)  the validity of \eqref{ip:norm2} for every $x\in \R$ and $y\in L^P(\R^+)$ satisfying \eqref{ip:stimax2} and \eqref{ip:stimay2}  implies  
the validity of \eqref{ip:norm} under condition \eqref{ip:stimax} and \eqref{ip:stimay} for $I=[0,n]$, $s^*=s_n^*$. Hence,
for every fixed \(n\ge \tau_3\) we have that  all the assumptions of Theorem \ref{thm:main1} are fulfilled (for \(p=1\)). As a consequence, we infer that 
 the boundary value problem
\eqref{PD} for \(I=[0,n]\), \(n\ge \tau_3\), admits 
 a solution \(x_n\in W^{1,1}([0,n])\)  satisfying \eqref{ip:stimax2} (with $x$ replaced by $x(t)$)  and \eqref{ip:stimay2} (with  $y(t)$ replaced by $x'(t)$).

Hence, put
$A_\infty^*:=\Phi^{-1}(\Phi(s_\infty^*)-2\ell_\infty)$, $B_\infty^*:=\Phi^{-1}(\Phi(s_\infty^*)+2\ell_\infty)$, 
 $C:=   k_{\infty}( |A_\infty^* |+|B_\infty^* |)$, $\mathcal{K}:=[A_\infty^*,B_\infty^* ]$, 
the assertion follows from Lemma \ref{lem:limite}.

\end{proof}

\bigskip
In the following existence result we do not require assumptions \eqref{ip:lip} and \eqref{ip:psiH}, provided that $\Phi$ is odd and supposing a strengthening of 
the other assumptions, requiring they are satisfied replacing \(s_\infty^*\) with \(s_T^*\), for a certain \(T>0\).

\begin{theorem}
\label{t:semir2}
Assume that \eqref{ip:kL1} and \eqref{ip:monot0} are satisfied. Moreover, suppose that $\Phi$ is odd in the (symmetric) 
interval $J^*$ and there exist a value $T>0$ such that $s_T^*\in J^*$ (see \eqref{eq:not}).

 Finally, assume that there exists a function  $\psi\in L^1(\R^+)$ such that 
 \beq \label{ip:bound3}        
   \Phi(s_T^*) \pm  2 \| \psi \|_{L^1(\R^+)} \in \Phi(J^*),
  \eeq    
 and
 \beq \label{ip:norm3}
    | f(t,x,y(t)) | \leq \psi(t) \qquad \text{a.e. in } \R^+, \eeq
 for every $x \in \R$ and $y\in L^1(\R^+)$ such that
\begin{align}
	\label{ip:stimax_T}
	|x|&<|\nu_1|+k_\infty\Phi^{-1}(\Phi(|s_T^*|) +2\|\psi\|_{L^1(\R^+)}); \\
	 \label{ip:stimay_T}
	|y(t)|&<\frac{1}{k(t)}\Phi^{-1}(\Phi(|s_T^*|) +2\|\psi\|_{L^1(\R^+)}) \qquad \text{for a.e. } t \in I.
\end{align}

Then, problem \eqref{PD-semir} admits a solution.
\end{theorem}

\begin{proof} As usual, we give the proof in the case $\Phi$ is increasing in $J^*$.

\noindent
First notice that for every $n>T$ large enough we have $s_n^*\in J^*$ and $s_n^*\le s_T^*$. Moreover, put $I_n:=[0,n]$, assumption  \eqref{ip:bound3} guarantees that \eqref{ip:bound} is fulfilled for $I=I_n$.
In fact, if we set $\ell_\infty:=  \| \psi \|_{L^1(\R^+)}$, if $\nu_2-\nu_1\ge  0$ we have
\[
	\inf J^*< \Phi(s_\infty^*) \pm 2\ell_\infty \le  \Phi(s_n^*) \pm 2\ell_\infty\le  \Phi(s_T^*)\pm  2\ell_\infty<\sup J^*.
\]

Instead, if $\nu_2-\nu_1<  0$ we have
\[
	\inf J^*< \Phi(s_T^*) \pm 2\ell_\infty <  \Phi(s_n^*)\pm 2\ell_\infty< \Phi(s_\infty^*)\pm 2\ell_\infty<\sup J^*.
\]

Finally, since $s_n^*\le s_T^*$, we have
\[\Phi^{-1}(\Phi(|s_n^*|) + 2 \|\psi\|_{L^1([0,n]}) \le 
 \Phi^{-1}(\Phi(|s_T^*|) + 2 \ell_\infty).\]

Therefore, the validity of \eqref{ip:norm3} under conditions \eqref{ip:stimax_T} and \eqref{ip:stimay_T} guarantees
the validity of \eqref{ip:norm} under conditions \eqref{ip:stimax_odd} and \eqref{ip:stimay_odd} (see Remark \ref{rem:odd}) for $I=I_n$. So,
   all the assumptions of Theorem \ref{thm:main1} are fulfilled for \(p=1\) and \(s^*=s_n^*\). As a consequence, 
 the boundary value problem
\eqref{PD} for \(I=[0,n]\) admits 
 a solution \(x_n\in W^{1,1}([0,n])\)  satisfying
\eqref{ip:estim} and \eqref{ip:unbound1} 
for $C:=  k_{\infty}\Phi^{-1}(\Phi(|s_T^*|) + 2 \ell_\infty)$, $\mathcal{K}:=[-\Phi^{-1}(\Phi(|s_T^*|) + 2 \ell_\infty),\Phi^{-1}(\Phi(|s_T^*|) + 2 \ell_\infty)]$, and then the assertion follows from Lemma \ref{lem:limite}. 
\end{proof}
\bigskip

We conclude the paper by providing some simple examples of b.v.p. on \([0,+\infty)\) where Theorems \ref{th:semir} and \ref{t:semir2} can be usefully applied.

\begin{example}
Consider the following problem
\beq \label{pr:ex1s} \begin{cases}((1+t^2)x'(t))' = t^2 g(x(t)) (x'(t))^3 \\
x(0)=0, \ x(+\infty)=\lambda. \end{cases}\eeq
where $g$ is a continuous bounded function, with \(\|g\|_{L^\infty(\R)}\le 1\). 

\smallskip
In this case \(\Phi(s)=s\), \(k(t)=1+t^2\). Then, assumptions  \eqref{ip:monot0}  and \eqref{ip:lip}  are trivially satisfied, as is \eqref{ip:kL1}, with
\[  L=1, \ J^*=\R,\ \Phi(J^*)=\R,\ k_{\infty} = \frac\pi{2}, \ s_\infty^*=\frac{2\lambda}{\pi}.\]

 Take the family of functions \(\psi_r(t):= r \min\{1,\frac{1}{t^2}\}\). Of course, \(\psi_r\in L^1(\R^+)\), with \(\|\psi_r\|_{L^1(\R^+)}=2r\), and assumption \eqref{ip:bound2} is trivially satisfied.
Moreover, assumption \eqref{ip:psiH} holds whenever \(|\lambda|<\frac{r\pi^2}{2}\).

\smallskip
Finally, put $r_0:=(\pi+4)^{-\frac32}$,  notice that if $y\in L^1(\R^+)$ satisfies
\[ k(t)|y(t)|\le |s_\infty^*|+ 2\|\psi_{r_0}\|_{L^1(\R^+)} = \frac{2|\lambda|}{\pi} +4r_0 \quad \text{a.e. } t\ge 0,
 \]
and \(|\lambda|<\frac{r_0\pi^2}{2}\), then we have
\begin{align*}|t^2 g(x)(y(t))^3|&\le  \frac{t^2}{(t^2+1)^3}\left(\frac{2|\lambda|}{\pi} +4r_0  \right)^3< \frac{t^2}{(t^2+1)^3} r_0^3 (\pi+4)^3 \\ & =\frac{t^2}{(t^2+1)^3} r_0  <\psi_{r_0}(t) \quad \text{a.e. } t\ge 0. \end{align*}
Therefore, \eqref{ip:norm2} is satisfied with $\psi=\psi_{r_0}$, 
for every $x\in \R$ and every $y\in L^1(\R^+)$ 
satisfying \eqref{ip:stimay2}.
Hence,  Theorem \ref{th:semir} ensures the existence of a solution to \eqref{pr:ex1s} for every \(|\lambda|<\frac{r_0\pi^2}{2}\).

\end{example}

The next example shows the utility of Theorem \ref{t:semir2}.

\begin{example}
Consider the following problem
\beq \label{pr:ex2s} \begin{cases}\left(\Phi((1+t^2)x'(t))\right)' = e^{-t} \arctan(x(t)x'(t)) \\
x(0)=0, \ x(+\infty)=\lambda \end{cases}\eeq 
where $\Phi$ is an odd operator, strictly monotone in a symmetric interval $J^*$. 

Notice that if we set $\psi(t):=\frac{\pi}{2} e^{-t}$, then assumption \eqref{ip:norm3} is trivially satisfied for every $x \in \R$ and $y\in L^1(\R^+)$. Therefore, Theorem \ref{t:semir2} can be applied provided that for some $T>0$ we have  $s_T^*\in J^*$ and \eqref{ip:bound3} holds. Hence, if $\Phi:\R \to \R$ is an increasing surjective operator, problem \eqref{pr:ex2s} admits solution for every $\lambda \in \R$. 

Instead, if $\Phi(J^*)=\R$ but $J^*\ne \R$  (as in the relativistic operator), it suffices to assume that  $s_\infty^*= \frac{2\lambda}{\pi}\in 
J^*$, since this implies that there exists $T>0$ such that $s_T^*\in J^*$ and the problem admits a solution.

Finally, if none of the above situations occur, we have to search for a suitable value $T$. For instance, put $T=1$, we have $s_T^*=\frac{4\lambda}{\pi}$. Then, if $\frac{4\lambda}{\pi}\in J^*$ 
and
\[\Phi \left(\frac{4\lambda}{\pi} \right)\pm \pi \in \Phi(J^*),\]
problem \eqref{pr:ex2s} has a solution.

\smallskip
Notice that for problem \eqref{pr:ex2s}, thanks to the boundedness of the right-hand side with respect to $x'$, Theorem \ref{t:semir2} provides a better result than Theorem \ref{th:semir}, since in order to have \eqref{ip:psiH} we have to take a function $\psi$ which is not the best dominating function, and this reduces the range of the values $\lambda$ for which the existence of a solution is guaranteed.

\end{example} 

\section*{Acknowledgments}

The authors are member of the {\sl Gruppo Nazionale per l'Analisi Matematica, la Probabilit\`a e le loro Applicazioni}
(GNAMPA) of the {\sl Istituto Nazionale di Alta Matematica} (INdAM). 

The first author is partially supported by the INdAM-GNAMPA project "Problemi variazionali: teoria cinetica, domini magnetici e non uniforme ellitticit\'a" CUP$\_$E5324001950001.

This paper is supported by the research project of MIUR (Italian Ministry of Education, University and Research) PRIN 2022- Progetti di Ricerca di rilevante Interesse Nazionale  {\it Nonlinear differential problems with applications to real phenomena} (Grant No. 2022ZXZTN2, CUP$\_$ I53D23002470006)

\end{document}